\documentclass[a4paper,reqno]{amsart}
\usepackage{amssymb, amsmath, amscd}
\def\pone{\hbox{\rm\textsf{1}\hspace*{-0.9ex}%
  \rule{0.15ex}{1.3ex}\hspace*{0.9ex}}}
\usepackage{amssymb, amsmath, amscd}
\def\mapright#1{\smash{\mathop{\longrightarrow}\limits\sp{#1}}}
\def\Id{\operatorname{Id}}
\newtheorem{theorem}{Theorem}[section]

\newtheorem{lemma}[theorem]{Lemma}

\makeatletter
 \@addtoreset{equation}{section}
\begin{document}

\title{(para)-K\"ahler Weyl structures}
\author{P. Gilkey and S. Nik\v cevi\'c}
\address{PG: Mathematics Department, University of Oregon, Eugene OR 97403 USA\\
   E-mail: gilkey@uoregon.edu}
\address{SN: Mathematical Institute, Sanu, Knez Mihailova 36, p.p. 367, 11001 Belgrade,
 Serbia\\ Email: stanan@mi.sanu.ac.rs}

\begin{abstract}{We work in both the complex
and in the para-complex categories and examine (para)-K\"ahler Weyl structures in both the geometric and in the algebraic settings. The higher
dimensional setting is quite restrictive. We show that any (para)-K\"ahler Weyl algebraic curvature tensor is in fact Riemannian in dimension
$m\ge6$; this yields as a geometric consequence that any (para)-K\"ahler Weyl geometric structure is trivial for $m\ge6$. By contrast, the
$4$-dimensional setting is, as always, rather special as it turns out that there are (para)-K\"ahler
Weyl algebraic curvature tensors which are not Riemannian if $m=4$. Since every (para)-K\"ahler Weyl
algebraic curvature tensor is geometrically realizable and since every $4$-dimensional Hermitian
manifold admits a unique (para)-K\"ahler Weyl structure, there are also non-trivial $4$-dimensional Hermitian (para)-K\"ahler Weyl manifolds.
}\end{abstract}

\maketitle

\section{Introduction}

Let $\nabla$ be a torsion free connection on a pseudo-Riemannian manifold $(M,g)$ of even dimension $m=2\bar m\ge4$. The triple $(M,g,\nabla)$ is
said to be a {\it Weyl structure} if there exists a smooth $1$-form $\phi$ so that $\nabla g=-2\phi\otimes g$. Such a geometric
structure was introduced by Weyl~\cite{W22} in an attempt to unify gravity
with electromagnetism. Although this approach failed for physical reasons, these geometries are still studied for their intrinsic interest
\cite{AI03,C01,I02,N05,O11}; they also appear in the mathematical physics literature
\cite{DT02,I00,M04}. Weyl geometry is relevant to submanifold geometry \cite{M02} and to contact geometry \cite{G09}. The
pseudo-Riemannian setting also is important
\cite{AR09,M01,S09} as are para-complex geometries
\cite{D06,FF11}. See also \cite{CP00,K11,PS91,PT93} for related results. 
The literature in the field is vast and we can only give a flavor of it for reasons of brevity. We shall be primarily interested in the
Hermitian setting. However since there are applications to higher signature geometry, we include the pseudo-Hermitian context as well; similarly we
treat para-Hermitian geometries as they can be studied with little additional effort.

Section~\ref{sect-1.1} of the Introduction deals with the real setting. In Theorem~\ref{thm-1}, we
recall the basic theorems of geometric realizability for affine, Riemannian, and Weyl curvature models and in
Theorem~\ref{thm-2} provide various characterizations of the notion of a trivial Weyl structure. 
Section~\ref{sect-1.2} treats the (para)-K\"ahler setting. In
Theorem~\ref{thm-3} we recall geometric realizibility results for (para)-K\"ahler affine and (para)-K\"ahler Riemannian
curvature models. Theorem~\ref{thm-4} presents results in the geometric setting for (para)-K\"ahler Weyl manifolds.
Theorem~\ref{thm-5} is one of the two main results of this paper: every (para)-K\"ahler
curvature model is geometrically realizable. The proof of Theorem~\ref{thm-5} relies on a curvature
decomposition result; the second main result of the paper, Theorem~\ref{thm-6}, discusses the space
of (para)-K\"ahler Weyl algebraic curvature tensors. 

\subsection{Riemannian, Affine, and Weyl geometry}\label{sect-1.1}
Let $(V,\langle\cdot,\cdot\rangle)$ be an inner product space of signature $(p,q)$ and dimension $m=p+q$; an inner product of signature $(0,4)$
is positive definite. A
$4$-tensor
$A\in\otimes^4V^*$ is said to be a {\it Riemannian algebraic curvature tensor} if $A$ satisfies the symmetries of the
Riemann curvature tensor, namely:
\begin{eqnarray}
&&A(x,y,z,w)+A(y,x,z,w)=0\,,\label{eqn-1}\\
&&A(x,y,z,w)+A(y,z,x,w)+A(z,x,y,w)=0\,,\label{eqn-2}\\
&&A(x,y,z,w)=A(w,z,x,y)\,.\label{eqn-3}
\end{eqnarray}
Let $\mathfrak{R}(V)$ be the subspace of
$\otimes^4V^*$ which consists of all tensors satisfying these relations. We say that a triple
$\mathcal{R}:=(V,\langle\cdot,\cdot\rangle,A)$ is a {\it Riemannian curvature model} if $A\in\mathfrak{R}(V)$.
One says that $\mathcal{R}$ is {\it geometrically
realizable} by a pseudo-Riemannian manifold if there is a point $P$ of some pseudo-Riemannian manifold $(M,g)$ and if
there is an isomorphism
$\Phi:V\rightarrow T_PM$ so: 
$$\Phi^*g_P=\langle\cdot,\cdot\rangle\quad\text{and}\quad \Phi^*R^g_P=A$$
where
$R^g$ is the curvature tensor of the Levi-Civita connection $\nabla^g$ on $M$. 

Affine differential geometry extends Riemannian geometry. A
pair $(M,\nabla)$ is said to be an {\it affine manifold} if $\nabla$ is a
torsion free connection on the tangent bundle $TM$. The curvature $R^\nabla$ of the connection
$\nabla$ then satisfies the identities of Equations~(\ref{eqn-1}) and (\ref{eqn-2}) but need no
longer satisfy Equation~(\ref{eqn-3}); if $A\in\otimes^4V^*$, one says $A$ is an
{\it affine algebraic curvature tensor} if $A$ satisfies
Equations~(\ref{eqn-1}) and (\ref{eqn-2}) and one lets $\mathfrak{A}(V)$ be
the set of all such tensors. Note that the corresponding {\it
curvature operator} $\hat A$ and the curvature tensor $A$ are related by the identity
$$\langle\hat A(x,y)z,w\rangle=A(x,y,z,w)\,.$$
The pair
$\mathcal{A}:=(V,A)$ is said to be an {\it affine curvature model} if $A\in\mathfrak{A}(V)$; such an $\mathcal{A}$ is
said to be {\it geometrically realizable by an affine manifold} if there is a
point
$P$ of some affine manifold $(M,\nabla)$ and if there is an isomorphism $\Phi:V\rightarrow T_PM$ so that
$\Phi^*R_P^\nabla=A$. 

Weyl geometry is in a sense midway between Riemannian geometry and affine geometry. A triple
$(M,g,\nabla)$ is said to be a {\it Weyl manifold} if $(M,g)$ is a pseudo-Riemannian manifold,
if $(M,\nabla)$ is an affine manifold, and if there exists a smooth $1$-form $\phi$ on $M$ so
that the structures are related by the equation:
\begin{equation}\label{eqn-4}
\nabla g=-2\phi\otimes g\,.
\end{equation}
Define the {\it Ricci-tensor} $\rho=\rho_\nabla$ and the {\it alternating Ricci tensor} $\rho_a=\rho_{a,\nabla}$ by:
\begin{eqnarray*}
&&\rho(x,y):=\operatorname{Tr}(z\rightarrow R(z,x)y),\\
&&\rho_a(x,y):=\textstyle\frac12\{\rho(x,y)-\rho(y,x)\}\,.
\end{eqnarray*}
There is an additional
curvature symmetry which pertains in Weyl geometry (see, for example, the discussion in \cite{GNS11}):
\begin{equation}\label{eqn-5}
R(x,y,z,w)+R(x,y,w,z)=-\textstyle\frac4m\rho_a(x,y)g(z,w)\,.
\end{equation}
The defining $1$-form $\phi$ is related to the curvature by the equation:
\begin{equation}\label{eqn-6}
d\phi=-\textstyle\frac1m\rho_a\,.
\end{equation}
Let
$\mathfrak{W}(V)\subset\otimes^4(V^*)$ be space of $4$-tensors satisfying
Equations~(\ref{eqn-1}), (\ref{eqn-2}), and (\ref{eqn-5}); these are the Weyl algebraic curvature tensors. If
$A\in\mathfrak{R}$, then $\rho_a=0$ and
$A(x,y,z,w)+A(x,y,w,z)=0$. Consequently:
$$\mathfrak{R}(V)\subset\mathfrak{W}(V)\subset\mathfrak{A}(V)\,.$$
A triple $\mathcal{W}:=(V,\langle\cdot,\cdot\rangle,A)$ is said to be a {\it Weyl curvature model} if
$A\in\mathfrak{W}(V)$. The notion of geometric realizability is
defined analogously in this setting.

We refer to \cite{BV07,BNGS06,GNS11} for the proof of the following result; the first two assertions are, of course, well
known:
\begin{theorem}\label{thm-1}
\ \begin{enumerate}
\item Every Riemannian curvature model is geometrically realizable by a pseudo-Riemannian manifold.
\item Every affine curvature model is geometrically realizable by an affine manifold.
\item Every Weyl curvature model is geometrically realizable by a Weyl manifold.
\end{enumerate}\end{theorem}

Weyl geometry is a conformal theory; if $g_1=e^{2f}g$ is conformally equivalent to $g$ and if $(M,g,\nabla)$ is a Weyl
manifold, then $(M,g_1,\nabla)$ is again a Weyl manifold with associated $1$-form $\phi_1$ given by
$\phi_1=\phi-df$. One has the following well known result characterizing trivial Weyl structures (see, for example,
\cite{GNS11}):
\begin{theorem}\label{thm-2}
Let $(M,g,\nabla)$ be a Weyl manifold with $H^1(M;\mathbb{R})=0$. The following assertions are
equivalent and if any is satisfied, then the Weyl structure is said to be trivial.
\begin{enumerate}
\item $d\phi=0$.
\item $\nabla=\nabla^{g_1}$ for some conformally equivalent metric $g_1$.
\item $R^\nabla\in\mathfrak{R}$.
\end{enumerate}
\end{theorem}

\subsection{K\"ahler geometry}\label{sect-1.2}
We now pass from the real to the (para)-complex setting. Let $V$ be a real vector space of even dimension $m=2\bar m$. A
complex structure on
$V$ is an endomorphism $J_-$ of
$V$ so
$J_-^2=-\Id$. Similarly, a para-complex structure on
$V$ is an endomorphism $J_+$ of $V$ so $J_+^2=\Id$ and $\operatorname{Tr}(J_+)=0$; this trace-free condition is automatic in
the complex setting but must be imposed in the para-complex setting. It is convenient to introduce the notation $J_\pm$ in
order to have a common formulation in both contexts although we shall never be considering both structures simultaneously.
In the geometric setting,
$(M,J_\pm)$ is said to be an {\it almost (para)-complex manifold} if $J_\pm$ is a smooth endomorphism of the tangent bundle
so that
$(T_PM,J_\pm)$ is a (para)-complex structure for every
$P\in M$. The almost (para)-complex structure $J_\pm$ is said to be {\it integrable} and the pair $(M,J_\pm)$ is said to be a
{\it (para)-complex manifold} if there are coordinate charts
$(x^1,y^1,...,x^{\bar m},y^{\bar m})$ covering $M$ so that:
\begin{equation}\label{eqn-7}
J_\pm\left\{\frac{\partial}{\partial x_i}\right\}
    =\frac{\partial}{\partial y_i}\quad\text{and}\quad
J_\pm\left\{\frac{\partial}{\partial y_i}\right\}
    =\pm\frac{\partial}{\partial
x_i}\quad\text{for}\quad 1\le i\le\bar m\,.
\end{equation}
If $(M,J_\pm)$ is an almost (para)-complex manifold and if $\nabla$ is a torsion free connection on $M$, then
$(M,J_\pm,\nabla)$ is said to be a {\it K\"ahler affine manifold} if $\nabla J_\pm=0$; this assumption then implies that
$J_\pm$ is integrable. The curvature satisfies an extra symmetry in this setting:
\begin{equation}\label{eqn-8}
R(x,y,z,w)=\mp R(x,y,J_\pm z,J_\pm w)\,.
\end{equation}
A (para)-complex
pseudo-Riemannian manifold
$(M,g,J_\pm)$ is said to be a {\it (para)-K\"ahler Hermitian} manifold if 
$J_\pm^*g=\mp g$ and $\nabla^gJ_\pm=0$.
Finally, a (para)-complex-Riemannian Weyl manifold $(M,g,J_\pm,\nabla)$ is said to be a {\it (para)-K\"ahler Weyl}
manifold if $\nabla J_\pm=0$.

We now pass to the algebraic context. Define the space of (para)-K\"ahler tensors $\mathfrak{K}_\pm$, the space of
(para)-K\"ahler affine algebraic curvature tensors
$\mathfrak{K}_{\pm,\mathfrak{A}}$, the space of (para)-K\"ahler Riemannian algebraic curvature tensors
$\mathfrak{K}_{\pm,\mathfrak{R}}$, and the space of (para)-K\"ahler Weyl algebraic curvature tensors 
$\mathfrak{K}_{\pm,\mathfrak{W}}$ by setting,
respectively:
\begin{eqnarray*}
&&\mathfrak{K}_\pm:=\{A\in\otimes^4V^*:A(x,y,z,w)=\mp A(x,y,J_\pm,z,J_\pm w)\},\\
&&\mathfrak{K}_{\pm,\mathfrak{A}}:=\mathfrak{K}_\pm\cap\mathfrak{A},\quad
  \mathfrak{K}_{\pm,\mathfrak{R}}:=\mathfrak{K}_\pm\cap\mathfrak{R},\quad
  \mathfrak{K}_{\pm,\mathfrak{W}}:=\mathfrak{K}_\pm\cap\mathfrak{W}\,.
\end{eqnarray*}

A triple
$\mathcal{K}_{\mathcal{A}}=(V,J_\pm,A)$ is said to be a {\it (para)-K\"ahler affine curvature model}
if
$(V,J_\pm)$ is (para)-complex and if $A\in\mathfrak{K}_{\pm,\mathfrak{A}}$. A quadruple
$\mathcal{K}_{\mathcal{R}}=(V,\langle\cdot,\cdot\rangle,J_\pm,A)$ is said to be a {\it (para)-K\"ahler Hermitian curvature
model} if
$J_\pm^*\langle\cdot,\cdot\rangle=\mp\langle\cdot,\cdot\rangle$ and if $A\in\mathfrak{K}_{\pm,\mathfrak{R}}$. A quadruple
$\mathcal{K}_{\mathcal{W}}=(V,\langle\cdot,\cdot\rangle,J_\pm,A)$ is said to be a {\it (para)-K\"ahler Weyl curvature
model} if
$J_\pm^*\langle\cdot,\cdot\rangle=\mp\langle\cdot,\cdot\rangle$ and if $A\in\mathfrak{K}_{\pm,\mathfrak{W}}$.

Let $\langle\cdot,\cdot\rangle$ have signature $(p,q)$; if $p=0$, then $\langle\cdot,\cdot\rangle$ is positive definite while
if $q=0$, then $\langle\cdot,\cdot\rangle$ is negative definite. In the
para-complex setting,
$p=q$ so
$\langle\cdot,\cdot\rangle$ is necessarily indefinite. In the complex setting, $p$ and $q$ must both be even; we emphasize
that we do not assume necessarily that the inner product is positive definite. We refer to \cite{BGN11} for the proof of
Assertion (1) and to \cite{BGM10} for the proof of Assertion (2) in the following result:
\begin{theorem}\label{thm-3}
\ \begin{enumerate}
\item Every (para)-K\"ahler affine curvature model is geometrically realizable by a (para)-K\"ahler affine
manifold.
\item Every (para)-K\"ahler Hermitian curvature model is geometrically realizable by a (para)-K\"ahler Hermitian manifold.
\end{enumerate}\end{theorem}

The (para)-K\"ahler form $\Omega_\pm$ is defined by the identity:
$$\Omega_\pm(x,y):=g(x,J_\pm y)\,.$$
Let $\delta$ be the co-derivative. We refer to \cite{PPS93,V82,V83} for the proof of Assertion (1) in the following result
in the positive definite setting -- the generalization to the indefinite setting is immediate. We refer
to
\cite{KK10} for the proof of Assertion (2) in the Riemannian setting -- the extension to the general
setting is immediate:

\begin{theorem}\label{thm-4}
\ \begin{enumerate}
\item Let $m\ge6$. If $(M,g,J_\pm,\nabla)$ is a 
(para)-K\"ahler  Weyl manifold, then the associated Weyl structure is trivial, i.e. locally
there is a conformally equivalent metric $g_1$ so that
$(M,g_1,J_\pm)$ is K\"ahler and so that $\nabla=\nabla^{g_1}$.
\item Every (para)-Hermitian manifold of dimension $4$ admits a unique 
(para)-K\"ahler Weyl structure defined by taking $\phi=\pm\frac12J_\pm^*\delta\Omega_\pm$.
\end{enumerate}
\end{theorem}

The following
theorem is the first main result of this paper:
\begin{theorem}\label{thm-5}
Every (para)-K\"ahler Weyl curvature model is geometrically realizable by a (para)-K\"ahler Weyl
manifold.
\end{theorem}   

Curvature decompositions play a central role in modern differential geometry. The following theorem is the second main
result of this paper and will play a central role in the proof of Theorem~\ref{thm-5}:

\begin{theorem}\label{thm-6}
Let $(V,\langle\cdot,\cdot\rangle,J_\pm)$ be a (para)-Hermitian vector
space.
\begin{enumerate}
\item If $m\ge6$, then
$\mathfrak{K}_{\pm,\mathfrak{W}}=\mathfrak{K}_{\pm,\mathfrak{R}}$.
\item If $m=4$, then
$\mathfrak{K}_{\pm,\mathfrak{W}}=\mathfrak{K}_{\pm,\mathfrak{R}}
   \oplus L^2_{0,\mp}$ where 
$$\rho_a:L^2_{0,\mp}\mapright{\approx}\Lambda^2_{0,\mp}:=
\{\Phi\in\Lambda^2(V^*):\Phi\perp\Omega_\pm\quad\text{and}\quad J_\pm^*\Phi=\mp\Phi\}\,.$$
\end{enumerate}
\end{theorem} 

Theorem~\ref{thm-6} is one of the facts about $4$-dimensional geometry that distinguishes it
from the higher dimensional setting; the module $L^2_{0,\mp}$ provides
additional curvature possibilities if
$m=4$. 

Curvature decompositions are fundamental in establishing geometrical realizability results. For example, we can use
Theorem~\ref{thm-6}~(1) to establish Theorem~\ref{thm-4}~(1) as follows. Suppose that $(M,g,J_\pm,\nabla)$ is a
(para)-K\"ahler Weyl manifold of dimension
$m\ge6$. By Theorem~\ref{thm-6}, $R^\nabla\in\mathfrak{K}_{\pm,\mathfrak{R}}\subset\mathfrak{R}$. By Theorem~\ref{thm-2},
there is a locally conformally equivalent metric $g_1$ so that
$\nabla=\nabla^{g_1}$; $g_1$ is globally defined if $H^1(M;\mathbb{R})=0$.

Here is a brief outline to the remainder of this paper. In Section~\ref{sect-2}, we review well known previous results
concerning curvature decompositions that we shall need. Theorem~\ref{thm-6} is established in Section~\ref{sect-3} and
Theorem~\ref{thm-5} is established in Section~\ref{sect-4}.

\section{Curvature decompositions}\label{sect-2}

In Section~\ref{sect-2.1}, the
structure groups $\mathcal{O}$, $\mathcal{U}_\pm$, and $\mathcal{U}_\pm^\star$ will be defined and the
fundamental facts needed from representation theory will be established. In Section~\ref{sect-2.2}, results
of Singer and Thorpe
\cite{ST69} giving the decomposition of
$\mathfrak{R}$ and results of Higa
\cite{H93,H94} giving the decomposition of $\mathfrak{W}$ as an $\mathcal{O}$-module will be presented.
In Section~\ref{sect-2.3} the Tricerri--Vanhecke decomposition \cite{TV81} of the space of Riemannian algebraic
curvature tensors
$\mathfrak{R}$ and the space of K\"ahler algebraic curvature tensors $\mathfrak{K}_{\pm,\mathfrak{R}}$
as $\mathcal{U}_\pm^\star$ modules will be outlined; this will rise to the decomposition of the space of Weyl algebraic
curvature tensors
$\mathfrak{W}$ as a
$\mathcal{U}_\pm^\star$ module. As we shall not need the decomposition of $\mathfrak{K}_{\pm,\mathfrak{A}}$ as a
$\mathcal{U}_\pm^\star$ module, we shall omit this decomposition and instead refer to the discussion in \cite{BGN11a}.

\subsection{Representation theory}\label{sect-2.1}
Let $(V,\langle\cdot,\cdot\rangle)$ be an inner product space. The orthogonal group $\mathcal{O}$ is the subgroup of
all invertible linear transformations of $V$ preserving the inner product.
If $(V,\langle\cdot,\cdot\rangle,J_\pm)$ is a (para)-Hermitian vector space, define:
\begin{eqnarray*}
&&\mathcal{U}_\pm:=\{T\in\mathcal{O}:TJ_\pm=J_\pm T\},\\
&&\mathcal{U}_\pm^\star:=\{T\in\mathcal{O}:TJ_\pm=J_\pm T\text{ or }
     TJ_\pm=-J_\pm T\}\,.
\end{eqnarray*}
It is convenient to work with the $\mathbb{Z}_2$ extensions
$\mathcal{U}_\pm^\star$ as we may then interchange the roles of
$J_\pm$ and $-J_\pm$. Let $\chi$ be the $\mathbb{Z}_2$ valued character of
$\mathcal{U}_\pm^\star$ so that:
$$
J_\pm T=\chi(T)TJ_\pm\quad\text{and}\quad T^*\Omega_\pm=\chi(T)\Omega_\pm
\quad\text{for}\quad T\in\mathcal{U}_\pm^\star\,.
$$
By an abuse of notation, we identify $\chi$ with the associated 1-dimensional module.
We can extend $\langle\cdot,\cdot\rangle$ to a natural non-degenerate inner product on $\otimes^kV$ and $\otimes^kV^*$. 
The following observation is fundamental in the subject:
\begin{lemma}\label{lem-1}
Let
$G\in\{\mathcal{O},\mathcal{U}_-,\mathcal{U}_-^\star,\mathcal{U}_+^\star\}$ and let
$\xi$ be a
$G$-submodule of
$\otimes^kV^*$. Then the restriction of the inner product on $\otimes^kV^*$ to $\xi$ is non-degenerate.
\end{lemma}

\begin{proof} Let $\{e_i\}$ be an orthonormal basis for $V$ and let $\{e^i\}$ be the associated dual basis for $V^*$. If
$I=(i_1,...,i_k)$ is a multi-index, set $e^I=e^{i_1}\otimes...\otimes e^{i_k}$. Then:
\begin{equation}\label{eqn-9}
(e^I,e^J):=\langle e^{i_1},e^{j_1}\rangle\cdot\cdot\cdot
\langle e^{i_k},e^{j_k}\rangle=\left\{\begin{array}{rl}
0&\text{ if }I\ne J\\
\pm1&\text{ if }I=J\end{array}\right\}\,.
\end{equation}

Let $Te_i=\langle e_i,e_i\rangle\cdot e_i$ define an element $T\in\mathcal{O}$. Suppose that $\xi$ is an $\mathcal{O}$
invariant subspace of $\otimes^kV^*$. Decompose $\xi=\xi_+\oplus\xi_-$ and
decompose $\otimes^kV^*=W_+\oplus W_-$ into the $\pm1$ eigenspaces of $T$. Since $T\in\mathcal{O}$, these decompositions are
orthogonal direct sums. By Equation~(\ref{eqn-9}), $W_+$ is spacelike and $W_-$ is timelike. Since
$\xi_\pm\subset W_\pm$, $\xi_+$ is spacelike and $\xi_-$ is timelike; the Lemma now follows in this
special case. If $G=\mathcal{U}_-$ or if
$G=\mathcal{U}_-^\star$, then we can choose the orthonormal basis so that 
$$J_-e_{2\nu-1}=e_{2\nu}\quad\text{and}\quad
  J_-e_{2\nu}=-e_{2\nu-1}\,.$$Since $J_-^*\langle\cdot,\cdot\rangle=\langle\cdot,\cdot\rangle$,
$J_-T=TJ_-$. Thus
$T\in G$ and the same argument pertains. Finally suppose $G=\mathcal{U}_+^\star$. We can choose the basis so
$$J_+e_{2\nu-1}=e_{2\nu}\quad\text{and}\quad  J_+e_{2\nu}=e_{2\nu-1}$$
where $e_{2\nu-1}$ is spacelike and $e_{2\nu}$ is
timelike. We now have
$T\in\mathcal{U}_+^\star-\mathcal{U}_+$.\end{proof}

We note that Lemma~\ref{lem-1} fails for the group $G=\mathcal{U}_+$. For example, let $V_\pm$ be the $\pm1$ eigenspaces
of
$J_+$; then $J_\pm V_\pm=V_\pm$ and $V_\pm$ is totally isotropic.
We can combine Lemma~\ref{lem-1} with same arguments as used in the positive definite setting to
establish the following result; we omit details in the interests of brevity:

\begin{lemma}\label{lem-2}
Let
$G\in\{\mathcal{O},\mathcal{U}_-,\mathcal{U}_-^\star,\mathcal{U}_+^\star\}$ and let
$\xi$ be a $G$-submodule of $\otimes^kV^*$.
\begin{enumerate}
\item There is an
orthogonal direct sum decomposition of
$\xi=\xi_1\oplus...\oplus\xi_k$ into irreducible $G$-submodules of $\xi$.
The multiplicity with which a given irreducible $G$-module $\eta$ appears in $\xi$ is
independent of the particular decomposition which is chosen. If $\xi_1$ appears with multiplicity
$1$ in the decomposition of
$\xi$ and if
$\eta$ is any $G$-submodule of $\xi$, then either $\xi_1\subset\eta$ or
$\xi_1\perp\eta$.
\item If $\xi_1\rightarrow\xi\rightarrow\xi_2$ is a short exact sequence of $G$-modules, then
$\xi$ is isomorphic to $\xi_1\oplus\xi_2$ as a $G$-module.
\end{enumerate}
\end{lemma}

We can illustrate Lemma~\ref{lem-2} as follows. Decompose
$$\otimes^2V^*=\Lambda^2(V^*)\oplus S^2(V^*)$$
as the direct sum of the alternating and the symmetric bilinear forms. We can further decompose
$\Lambda^2(V^*)=\Lambda^2_\pm\oplus\chi\oplus\Lambda^2_{0,\mp}$ and
$S^2(V^*)=S_\pm^2\oplus\pone\oplus S^2_{0,\mp}$
where
$$\begin{array}{ll}
\Lambda^2_\pm:=\{\omega\in\Lambda^2:J_\pm^*\omega=\pm\omega\},&
\chi:=\Omega_\pm\cdot\mathbb{R},\\
\Lambda^2_{0,\mp}:=\{\omega\in\Lambda^2:J_\pm^*\omega=\mp\omega,\ \omega\perp\Omega_\pm\},
\vphantom{\vrule height 11pt}\\
S_\pm^2:=\{\theta\in S^2:J_\pm^*\theta=\pm\theta\},&
\pone:=\langle\cdot,\cdot\rangle\cdot\mathbb{R},
\vphantom{\vrule height 11pt}\\
S^2_{0,\mp}:=\{\theta\in S^2:J_\pm^*\theta=\mp\theta,\ \theta\perp\langle\cdot,\cdot\rangle\}.
\vphantom{\vrule height 11pt}
\end{array}$$
\begin{lemma}\label{lem-3}
Let $(V,\langle\cdot,\cdot\rangle,J_\pm)$ be a (para)-Hermitian vector space. We have the
following decomposition of $\Lambda^2(V^*)$, $S^2(V^*)$, and
$\otimes^2V^*$ into inequivalent and irreducible $\mathcal{U}_\pm^\star$ modules:
\smallbreak
\qquad\qquad$\Lambda^2(V^*)=\Lambda^2_\pm\oplus\chi\oplus\Lambda^2_{0,\mp}\,,\qquad S^2(V^*)=S_\pm^2\oplus\pone\oplus
S^2_{0,\mp}$,
\smallbreak\qquad\qquad$\otimes^2V^*=\Lambda^2_\pm\oplus\chi\oplus\Lambda^2_{0,\mp}\oplus S_\pm^2
\oplus\pone\oplus S^2_{0,\mp}$.
\end{lemma}

We note that $\Lambda^2_{0,\mp}$ and $S^2_{0,\mp}$ are isomorphic $U_\pm$
modules, that $\Lambda^2_{0,\mp}$ is isomorphic to $S^2_{0,\mp}\otimes\chi$ as a $U_\pm^\star$ module,
and that $\Lambda^2_+$ is not an irreducible $U_+$ module. We complete our discussion of elementary
representation theory with the following diagonalization result (see, for example, the discussion in
\cite{BGN12}):

\begin{lemma}\label{lem-4}
If $\xi$ is a non-trivial proper $\mathcal{U}_\pm^\star$ submodule of 
$\Lambda^2_\pm\oplus\Lambda^2_\pm$, then there
exists
$(a,b)\ne(0,0)$ so 
$\xi=\xi(a,b):=\{(a\theta,b\theta)\}_{\theta\in\Lambda^2_\pm}
\subset\Lambda^2_\pm\oplus\Lambda^2_\pm$.
\end{lemma}

\subsection{The Singer--Thorpe and the Higa decompositions}\label{sect-2.2}
We now examine the $\mathcal{O}$-module structure of $\mathfrak{R}$
and $\mathfrak{W}$. Let
$$
   S_0^2:=\{\theta\in S^2:\theta\perp\langle\cdot,\cdot\rangle\}\quad\text{and}\quad
  \mathfrak{C}:=\ker\{\rho\}\cap\mathfrak{R}
$$ 
be the
$\mathcal{O}$ modules of trace free symmetric
$2$-tensors and {\it Weyl conformal
curvature tensors}, respectively. We refer to Singer and Thorpe \cite{ST69} for
the proof of Assertion (1) and to Higa
\cite{H93,H94} for the proof of Assertion (2) in the following result:
\begin{theorem}\label{thm-7}
Let $n\ge4$. 
\begin{enumerate}
\item We may decompose $\mathfrak{R}=\pone\oplus S_0^2\oplus\mathfrak{C}$ as the orthogonal direct sum of
irreducible and inequivalent $\mathcal{O}$ modules.
\item We may decompose $\mathfrak{W}=\pone\oplus S_0^2\oplus\mathfrak{C}\oplus\mathfrak{P}$ as the orthogonal
direct sum of irreducible and inequivalent $\mathcal{O}$ modules. Here $\rho_a$ provides an $\mathcal{O}$ module isomorphism
from $\mathfrak{P}$ to $\Lambda^2$ with the inverse embedding
$\Xi:\Lambda^2\mapright{\approx}\mathfrak{P}\subset\mathfrak{W}$ given by:
\begin{equation}\label{eqn-10}
\begin{array}{l}
\Xi(\psi)(x,y,z,w):=2\psi(x,y)\langle z,w\rangle+\psi(x,z)\langle
y,w\rangle-\psi(y,z)\langle x,w\rangle\\
\phantom{\Xi(\psi)(x,y,z,w):}-\psi(x,w)\langle y,z\rangle+\psi(y,w)\langle
x,z\rangle\,.\vphantom{\vrule height 11pt}
\end{array}\end{equation}
\end{enumerate}\end{theorem}

\subsection{The Tricerri-Vanhecke decompositions}\label{sect-2.3}
The following
decompositions of $\mathfrak{R}$ and $\mathfrak{K}_{\pm,\mathfrak{R}}$ as $\mathcal{U}_-$ modules was
given by Tricerri and Vanhecke \cite{TV81} in the positive definite setting; they extend easily to the
more general context
\cite{BGM10,BGN11}. The decomposition of $\mathfrak{W}$ as a
$\mathcal{U}_\pm^\star$ module then follows from Lemma~\ref{lem-3} and Theorem~\ref{thm-7}. 

\begin{theorem}\label{thm-8}
Let $(V,\langle\cdot,\cdot\rangle,J_\pm)$ be a (para)-Hermitian vector
space. We have the following decompositions of $\mathfrak{R}$, 
$\mathfrak{K}_{\pm,\mathfrak{R}}$, and $\mathfrak{W}$ as $\mathcal{U}_\pm^\star$
modules:
\begin{equation}\begin{array}{l}
  \mathfrak{R}=W_{\pm,1}\oplus...\oplus W_{\pm,10},\\
  \mathfrak{K}_{\pm,\mathfrak{R}}=W_{\pm,1}\oplus  W_{\pm,2}\oplus  W_{\pm,3},\label{eqn-11}
     \vphantom{\vrule height 11pt}\\
  \mathfrak{W}=W_{\pm,1}\oplus...\oplus W_{\pm,13}\,.\vphantom{\vrule height 11pt}
\end{array}\end{equation}
If $n=4$, we omit the modules 
$\{W_{\pm,5},W_{\pm,6},W_{\pm,10}\}$. If $n=6$, 
we omit the module $ W_{\pm,6}$. The decomposition of
Equation~(\ref{eqn-11}) is then into irreducible $\mathcal{U}_\pm^\star$ modules. We have $\mathcal{U}_\pm^\star$ module
isomorphisms:
\begin{eqnarray}
&&W_{\pm,1}\approx W_{\pm,4}\approx\pone,\quad
      W_{\pm,2}\approx W_{\pm,5}\approx S^2_{0,\mp},\quad
      W_{\pm,9}\approx W_{\pm,13}\approx\Lambda^2_\pm,\label{eqn-11a}\\
&&W_{\pm,8}\approx S_\pm^2,\quad W_{\pm,11}\approx\chi,\quad
W_{\pm,12}\approx\Lambda^2_{0,\mp}\,.\label{eqn-11b}
\end{eqnarray}
With exception of the isomorphisms described in Equation~(\ref{eqn-11a}), these are
inequivalent
$\mathcal{U}_\pm^\star$ modules. The isomorphism $\Psi$ from $\Lambda^2_\pm$ to $W_{\pm,9}$ is given by
setting 
\begin{eqnarray}
&&\Psi(\psi)(x,y,z,w):=2\langle x,J_\pm y\rangle\psi(z,J_\pm w)+2\langle z,J_\pm
w\rangle\psi(x,J_\pm y)\nonumber\\ 
&&\qquad\qquad\qquad\qquad+\langle x,J_\pm z\rangle\psi(y,J_\pm
w)+\langle y,J_\pm w\rangle\psi(x,J_\pm z)\label{eqn-12}\\ 
&&\qquad\qquad\qquad\qquad-\langle
x,J_\pm w\rangle\psi(y,J_\pm z)-\langle y,J_\pm z\rangle\psi(x,J_\pm w)\,.\nonumber
\end{eqnarray}
\end{theorem}

It is worth describing the some of these in a bit more detail. Let $\{e_i\}$ be a basis for $V$. Set
$\varepsilon_{ij}:=\langle e_i,e_j\rangle$. Define $\rho_{J_\pm}(x,y):=\varepsilon^{il}A(e_i,x,J_\pm y,J_\pm e_l)$. We then have:
\begin{eqnarray*}
&&W_{\pm,7}=\{A\in\mathfrak{R}:A(J_\pm x,y,z,w)=A(x,y,J_\pm z,w)\},\\
&&W_{\pm,3}=\mathfrak{K}_{\pm,\mathfrak{R}}\cap\ker(\rho),\\
&&W_{\pm,6}=\{A\in\mathfrak{R}:J_\pm^*A=A\}\cap\{\mathfrak{K}_{\pm,\mathfrak{R}}\}^\perp\cap\{W_{\pm,7}\}^\perp
\cap\ker(\rho\oplus\rho_{J_\pm}),\\
&&W_{\pm,10}=\{A\in\mathfrak{R}:J_\pm^*A=-A\}\cap\ker(\rho\oplus\rho_{J_\pm})\,.
\end{eqnarray*}

\section{The proof of Theorem~\ref{thm-6}}\label{sect-3}
If $\eta$ is an irreducible
$\mathcal{U}_\pm^\star$ module and if $\xi$ is a submodule of $\otimes^4V^*$, let $n_\eta(\xi)$ be the
multiplicity with which $\eta$ appears in the decomposition of $\xi$ given in Lemma~\ref{lem-2}; 
note that $W_{\pm,4}\approx W_{\pm,1}$ and $W_{\pm,2}\approx W_{\pm,5}$. We apply
Theorem~\ref{thm-8}. If
$\eta$ is isomorphic to $W_{i,\pm}$ for $i\in\{1,2,3,4,5,6,7,8,10\}$, then $n_\eta(\Lambda^2)=0$ so:
$$n_\eta(\mathfrak{K}_{\pm,\mathfrak{W}})
=n_\eta(\mathfrak{K}_{\pm,\mathfrak{R}})=\left\{\begin{array}{lll}
1&\text{if }i=1,2,3,4,5\\
0&\text{if }i=6,7,8,10\end{array}\right\}\,.$$
Thus only the multiplicities of the representations $\{\chi,\Lambda^2_{0,\mp},\Lambda^2_\pm\}$ are at
issue.

\subsection{The module $\chi=\Xi(\Omega_\pm)$ for $m\ge4$}
Let $\{e_i\}$ be an orthonormal basis for $V$ with $J_\pm e_{2i-1}=e_{2i}$ and $J_\pm e_{2i}=\pm
e_{2i-1}$. Let $\varepsilon_{ij}:=\langle e_i,e_j\rangle$.
We use Equation~(\ref{eqn-10}) to see:
\begin{eqnarray*}
&&\phantom{\mp}\Xi(\Omega_\pm)(e_1,e_4,e_3,e_1)
     =-\langle e_4,J_\pm e_3\rangle\langle e_1,e_1\rangle
     =-\varepsilon_{11}\varepsilon_{44},\\ 
&&\mp\Xi(\Omega_\pm)(e_1,e_4,J_\pm e_3,J_\pm e_1)
     =\pm\langle e_1,J_\pm J_\pm e_1\rangle\langle e_4,J_\pm e_3\rangle
     =\varepsilon_{11}\varepsilon_{44}\,.
\end{eqnarray*}
Thus $\Xi(\Omega_\pm)$ does not satisfy the K\"ahler identity given in Equation~(\ref{eqn-8}). Consequently,
$n_\chi(\mathfrak{K}_{\pm,\mathfrak{W}})=0$.
\subsection{The module $W_{\pm,12}=\Xi(\Lambda^2_{0,\mp})$ for $m\ge6$}\label{sect-3.2}
Set
$$\psi_{0,\pm}:=e^1\otimes e^2-e^2\otimes e^1-\varepsilon_{11}\varepsilon_{33}\{e^3\otimes
e^4-e^4\otimes e^3\}\,.$$
Clearly
$\psi_{0,\pm}\perp\Omega_\pm$. Since $J_\pm^*\psi_{0,\pm}=\mp\psi_{0,\pm}$, $\psi_{0,\pm}\in
\Lambda^2_{0,\mp}$. 
By Equation (\ref{eqn-10}):
\begin{eqnarray*}
&&\phantom{\mp}\Xi(\psi_{0,\pm})(e_5,e_1,e_2,e_5)
   =-\psi_{0,\pm}(e_1,e_2)\langle e_5,e_5\rangle=-\varepsilon_{55},\\
&&\mp\Xi(\psi_{0,\pm})(e_5,e_1,J_\pm e_2,J_\pm e_5)
=0\,.
\end{eqnarray*}
Consequently $\Xi(\psi_{0,\pm})$ does not satisfy the K\"ahler identity and we conclude that
$n_{\Lambda^2_{0,\mp}}(\mathfrak{K}_{\pm,\mathfrak{W}})=0$ if $m\ge6$.

\subsection{The module $\Lambda^2_{0,\mp}$ if $m=4$}\label{sect-3.3}
The argument given above in Section~\ref{sect-3.2} does not, of course, pertain if $m=4$ since we can not examine 
$\Xi(\psi_{0,\pm})(e_5,e_1,e_2,e_5)$.
Let $\eta=\Lambda^2_{0,\mp}$. As noted above, $n_\eta(\mathfrak{K}_{\pm,\mathfrak{W}})\le1$. Thus if we can exhibit
a non-trivial element of $W_{\pm,12}\cap\mathfrak{K}_{\pm,\mathfrak{W}}$, we will have
$n_\eta(\mathfrak{K}_{\pm,\mathfrak{W}})=1$. We work in the positive definite setting for the
moment to simplify the argument. Let
\begin{eqnarray*}
&&\psi_{0,+}:=e^1\otimes e^2-e^2\otimes e^1-e^3\otimes e^4+e^4\otimes e^3,\\
&&\langle\cdot,\cdot\rangle:=e^1\otimes e^1+e^2\otimes e^2+e^3\otimes e^3+e^4\otimes e^4\,.
\end{eqnarray*}
Decompose
$A:=\Xi(\psi_{0,+})=A_1+A_2+A_3+A_4+A_5$ using the notation of Equation~(\ref{eqn-10}) where
$$\begin{array}{rr}
A_1(x,y,z,w):=2\psi_{0,+}(x,y)\langle z,w\rangle,&A_2(x,y,z,w):=\psi_{0,+}(x,z)\langle y,w\rangle,\\
A_3(x,y,z,w):=-\psi_{0,+}(y,z)\langle x,w\rangle,&\qquad A_4(x,y,z,w):=-\psi_{0,+}(x,w)\langle
y,z\rangle,\vphantom{\vrule height 11pt}\\ A_5(x,y,z,w):=\psi_{0,+}(y,w)\langle
x,z\rangle\,.\vphantom{\vrule height 11pt}
\end{array}$$
As a short hand, we set $e^{ijkl}:=e^i\otimes e^j\otimes e^k\otimes e^l$. We may then express:
\medbreak\quad
$A_1=2e^{1211}+2e^{1222}+2e^{1233}+2e^{1244}-2e^{2111}-2e^{2122}-2e^{2133}-2e^{2144}$
\smallbreak\qquad\quad
$-2e^{3411}-2e^{3422}-2e^{3433}-2e^{3444}+2e^{4311}+2e^{4322}+2e^{4333}+2e^{4344}$,
\smallbreak\quad
$A_2=e^{1121}+e^{1222}+e^{1323}+e^{1424}-e^{2111}-e^{2212}-e^{2313}-e^{2414}$
\medbreak\qquad\quad
$-e^{3141}-e^{3242}-e^{3343}-e^{3444}+e^{4131}+e^{4232}+e^{4333}+e^{4434}$,
\smallbreak\quad
$A_3=-e^{1121}-e^{2122}-e^{3123}-e^{4124}+e^{1211}+e^{2212}+e^{3213}+e^{4214}$
\medbreak\qquad\quad
$+e^{1341}+e^{2342}+e^{3343}+e^{4344}-e^{1431}-e^{2432}-e^{3433}-e^{4434}$,
\smallbreak\quad
$A_4=-e^{1112}-e^{1222}-e^{1332}-e^{1442}+e^{2111}+e^{2221}+e^{2331}+e^{2441}$
\smallbreak\qquad\quad
$+e^{3114}+e^{3224}+e^{3334}+e^{3444}-e^{4113}-e^{4223}-e^{4333}-e^{4443}$,
\smallbreak\quad
$A_5=e^{1112}+e^{2122}+e^{3132}+e^{4142}-e^{1211}-e^{2221}-e^{3231}-e^{4241}$
\smallbreak\qquad\quad
$-e^{1314}-e^{2324}-e^{3334}-e^{4344}+e^{1413}+e^{2423}+e^{3433}+e^{4443}$.
\medbreak\noindent We may ignore the terms in $A_1$ as these belong to $\mathfrak{K}_{+}$. The
remaining terms yield a tensor which is anti-symmetric both in the first two and in the last two
indices. Thus automatically terms of the form
$e^{**12}$ or $e^{**34}$ will belong to $\mathfrak{K}_+$ and can be ignored. Using the $\mathbb{Z}_2$ symmetry, we may consider
terms $e^{ijkl}$ where $i<j$ and $k<l$. We establish the K\"ahler identity and show that
$n_\eta(\mathfrak{K}_{+,\mathfrak{W}})=1$ if $m=4$ in the positive definite setting by examining the
following crucial terms:
\smallbreak\centerline{$\begin{array}{|r|r||r|r|}
\noalign{\hrule}\text{Term}&\text{Coeff.}&\text{Term}&\text{Coeff.}\\
\noalign{\hrule}\noalign{\hrule}e^{1323}&A_2=1&e^{1314}&A_5=-1\vphantom{\vrule height 11pt}\\
\noalign{\hrule}e^{1424}&A_2=1&e^{1413}&A_5=1\vphantom{\vrule height 11pt}\\
\noalign{\hrule}e^{2313}&A_2=-1&e^{2324}&A_5=-1\vphantom{\vrule height 11pt}\\
\noalign{\hrule}e^{2414}&A_2=-1&e^{2423}&A_5=1\vphantom{\vrule height 11pt}\\
\noalign{\hrule}
\end{array}$}

\medbreak We now complexify and let $W:=V\otimes_{\mathbb{R}}\mathbb{C}$. 
Extend $\langle\cdot,\cdot\rangle$, $J_-$, and
$A$ to be complex bilinear, complex linear, and complex multi-linear, respectively.
Let:
$$V_{2,2}:=\operatorname{Span}_{\mathbb{R}}\{\sqrt{-1}e_1,\sqrt{-1}e_2,e_3,e_4\}\,.$$
Then $(\langle\cdot,\cdot\rangle,J_-)$ restricts to a pseudo-Hermitian almost complex structure on
$V_{2,2}$ of signature $(2,2)$. Note that
\begin{eqnarray*}
&&\operatorname{Re}(A|_{V_{2,2}})\in W_{\pm,12}(V_{2,2})\cap\mathfrak{K}_{-,\mathfrak{W}}(V_{2,2}),\\
&&\operatorname{Im}(A|_{V_{2,2}})\in W_{\pm,12}(V_{2,2})\cap\mathfrak{K}_{-,\mathfrak{W}}(V_{2,2})\,.
\end{eqnarray*}
Since $A|_{V_{2,2}}\ne0$, at least one of these tensors is non-trivial and the desired conclusion
follows for neutral signature $(2,2)$; a similar argument applied to
$$V_{4,0}:=\operatorname{Span}_{\mathbb{R}}\{\sqrt{-1}e_1,\sqrt{-1}e_2,\sqrt{-1}e_3,\sqrt{-1}e_4\}$$
establishes the desired
result in signature $(4,0)$ (which is the negative definite setting). Finally, by considering 
$$U_{2,2}:=\operatorname{Span}_{\mathbb{R}}\{e_1,\sqrt{-1}e_2,e_3,\sqrt{-1}e_4\}$$
and $J_+:=\sqrt{-1}J_-$, we can construct an example in the para-complex setting.

\subsection{The module $\Lambda^2_\pm$ if $m\ge6$}\label{sect-3.4}
Let $\eta=\Lambda^2_\pm$. Then $W_{\pm,9}\oplus W_{\pm,13}\approx 2\cdot \eta$.
We adopt the
notation of Equation (\ref{eqn-10}) and of Equation (\ref{eqn-12}). For $(a,b)\ne(0,0)$, let
$$\xi(a,b):=\operatorname{Range}\{a\Xi+b\Psi\}\subset W_{\pm,9}\oplus W_{\pm,13}\,.$$
By
Lemma~\ref{lem-4}, every non-trivial proper submodule of
$W_{\pm,9}\oplus W_{\pm,13}$ is isomorphic to $\xi(a,b)$ for some $(a,b)\ne0$. We suppose 
$\xi(a,b)\subset\mathfrak{K}_{\pm,\mathfrak{W}}$ and thus
$$
(a\Xi+b\Psi)\psi\in\mathfrak{K}_{\pm,\mathfrak{W}}\quad\text{for all}\quad\psi_\pm\in\Lambda^2_\pm\,.
$$ 
Set
$\psi_\pm:=e^1\otimes e^3-e^3\otimes e^1\pm e^2\otimes e^4\mp e^4\otimes e^2$.
Then $J_\pm^*\psi_\pm=\pm\psi_\pm$ so $\psi_\pm\in\Lambda^2_\pm$. We show that $b=0$ by checking:
\begin{eqnarray*}
&&a\Xi(\psi_\pm)(e_5,e_6,e_1,e_4)=0,\\
&&\mp a\Xi(\psi_\pm)(e_5,e_6,J_\pm e_1,J_\pm e_4)=0,\\
&&b\Psi(\psi_\pm)(e_5,e_6,e_1,e_4)
   =2b\langle e_5,J_\pm e_6\rangle\psi_\pm(e_1,J_\pm e_4)=2b\varepsilon_{55},\\
&&\mp b\Psi(\psi_\pm)(e_5,e_6,J_\pm e_1,J_\pm e_4)=\mp2b\langle e_5,J_\pm e_6\rangle\psi_\pm(J_\pm
e_1,J_\pm J_\pm e_4)=-2b\varepsilon_{55}\,.
\end{eqnarray*}
We show that $a=0$ and complete the proof of
Theorem~\ref{thm-6} if $m\ge6$ by checking:
\begin{eqnarray*}
&&a\Xi(\psi_\pm)(e_5,e_1,e_3,e_5)=-a\psi_\pm(e_1,e_3)\langle e_5,e_5\rangle=-a\varepsilon_{55},\\
&&\mp a\Xi(\psi_\pm)(e_5,e_1,e_4,e_6)=0\,.
\end{eqnarray*}

\subsection{The module $\Lambda^2_\pm$ if $m=4$}
Again, the argument given in Section~\ref{sect-3.4} is not available if $m=4$ since, for example, we can
not examine
$(e_5,e_1,e_4,e_6)$. Let $\eta=\Lambda^2_\pm$. Again, we first work in the positive definite setting.
Since
$n_\eta(\mathfrak{K}_{+,\mathfrak{R}})=0$, to show $n_\eta(\mathfrak{K}_{+,\mathfrak{W}})=1$ it suffices
to construct a suitable element of
$\mathfrak{K}_{+,\mathfrak{W}}$.  Let 
\begin{eqnarray*}
&&\psi_-:=e^1\otimes e^3-e^3\otimes e^1-e^2\otimes e^4+e^4\otimes e^2\in\Lambda^2_-,\\
&&\langle\cdot,\cdot\rangle:=e^1\otimes e^1+e^2\otimes e^2+e^3\otimes e^3+e^4\otimes e^4\,.
\end{eqnarray*}
Adopt the notation of Equation~(\ref{eqn-10}) to decompose $\Xi(\psi_-)=F+G+H+J+K$ where
$$\begin{array}{rr}
F(x,y,z,w):=2\psi_-(x,y)\langle z,w\rangle,&G(x,y,z,w):=\psi_-(x,z)\langle y,w\rangle,\\
H(x,y,z,w):=-\psi_-(y,z)\langle x,w\rangle,&J(x,y,z,w):=-\psi_-(x,w)\langle y,z\rangle ,\\
K(x,y,z,w):=\psi_-(y,w)\langle x,z\rangle\,.
\end{array}$$
We compute:
\medbreak\quad
$F=+2e^{1311}+2e^{1322}+2e^{1333}+2e^{1344}-2e^{3111}-2e^{3122}-2e^{3133}-2e^{3144}$
\smallbreak\qquad\quad
$+2e^{4211}+2e^{4222}+2e^{4233}+2e^{4244}-2e^{2411}-2e^{2422}-2e^{2433}-2e^{2444}$,
\medbreak\quad
$G=+e^{1131}+e^{1232}+e^{1333}+e^{1434}-e^{3111}-e^{3212}-e^{3313}-e^{3414}$
\smallbreak\qquad\quad
$+e^{4121}+e^{4222}+e^{4323}+e^{4424}-e^{2141}-e^{2242}-e^{2343}-e^{2444}$,
\medbreak\quad
$H=-e^{1131}-e^{2132}-e^{3133}-e^{4134}+e^{1311}+e^{2312}+e^{3313}+e^{4314}$
\smallbreak\qquad\quad
$-e^{1421}-e^{2422}-e^{3423}-e^{4424}+e^{1241}+e^{2242}+e^{3243}+e^{4244}$,
\medbreak\quad
$J=-e^{1113}-e^{1223}-e^{1333}-e^{1443}+e^{3111}+e^{3221}+e^{3331}+e^{3441}$
\smallbreak\qquad\quad
$-e^{4112}-e^{4222}-e^{4332}-e^{4442}+e^{2114}+e^{2224}+e^{2334}+e^{2444}$,
\medbreak\quad
$K=+e^{1113}+e^{2123}+e^{3133}+e^{4143}-e^{1311}-e^{2321}-e^{3331}-e^{4341}$
\smallbreak\qquad\quad
$+e^{1412}+e^{2422}+e^{3432}+e^{4442}-e^{1214}-e^{2224}-e^{3234}-e^{4244}$.
\medbreak\noindent
Next we examine the role of $\Psi$.
Set $\tilde \varepsilon(x,y):=\langle x,Jy\rangle $ and $\tilde\psi_-(x,y):=\psi_-(x,Jy)$. We expand 
$\Psi(\psi_-)=R+S+T+U+V+W$
where
$$\begin{array}{rr}
R(x,y,z,w):=2\tilde \varepsilon(x,y)\tilde\psi_-(z,w),&S(x,y,z,w):=2\tilde \varepsilon(z,w)\tilde\psi_-(x,y),\vphantom{\vrule
heigh 11pt}\\ T(x,y,z,w):=\tilde \varepsilon(x,z)\tilde\psi_-(y,w),&U(x,y,z,w):=\tilde
\varepsilon(y,w)\tilde\psi_-(x,z),\vphantom{\vrule height 11pt}\\ V(x,y,z,w):=-\tilde
\varepsilon(x,w)\tilde\psi_-(y,z),&W(x,y,z,w):=-\tilde \varepsilon(y,z)\tilde\psi_-(x,w)\,.\vphantom{\vrule height 11pt}
\end{array}$$
We compute:
\medbreak\quad
$\tilde\varepsilon=-e^1\otimes e^2+e^2\otimes e^1-e^3\otimes e^4+e^4\otimes e^3$,
\medbreak\quad
$\tilde\psi_-=-e^1\otimes e^4+e^4\otimes e^1-e^2\otimes e^3+e^3\otimes e^2$,
\medbreak\quad
$R=-2e^{1241}+2e^{2141}-2e^{3441}+2e^{4341}+2e^{1214}-2e^{2114}+2e^{3414}-2e^{4314}$
\smallbreak\qquad\quad
$-2e^{1232}+2e^{2132}-2e^{3432}+2e^{4332}+2e^{1223}-2e^{2123}+2e^{3423}-2e^{4323}$,
\medbreak\quad
$S=-2e^{4112}+2e^{4121}-2e^{4134}+2e^{4143}+2e^{1412}-2e^{1421}+2e^{1434}-2e^{1443}$
\smallbreak\qquad\quad
$-2e^{3212}+2e^{3221}-2e^{3234}+2e^{3243}+2e^{2312}-2e^{2321}+2e^{2334}-2e^{2343}$,
\medbreak\quad
$T=-e^{1421}+e^{2411}-e^{3441}+e^{4431}+e^{1124}-e^{2114}+e^{3144}-e^{4134}$
\smallbreak\qquad\quad
$-e^{1322}+e^{2312}-e^{3342}+e^{4332}+e^{1223}-e^{2213}+e^{3243}-e^{4233}$,
\medbreak\quad
$U=-e^{4112}+e^{4211}-e^{4314}+e^{4413}+e^{1142}-e^{1241}+e^{1344}-e^{1443}$
\smallbreak\qquad\quad
$-e^{3122}+e^{3221}-e^{3324}+e^{3423}+e^{2132}-e^{2231}+e^{2334}-e^{2433}$,
\medbreak\quad
$V=e^{1412}-e^{2411}+e^{3414}-e^{4413}-e^{1142}+e^{2141}-e^{3144}+e^{4143}$
\smallbreak\qquad\quad
$+e^{1322}-e^{2321}+e^{3324}-e^{4323}-e^{1232}+e^{2231}-e^{3234}+e^{4233}$,
\medbreak\quad
$W=e^{4121}-e^{4211}+e^{4341}-e^{4431}-e^{1124}+e^{1214}-e^{1344}+e^{1434}$
\smallbreak\qquad\quad
$+e^{3122}-e^{3212}+e^{3342}-e^{3432}-e^{2123}+e^{2213}-e^{2343}+e^{2433}$.
\medbreak
We may ignore the $F$ and the $S$ terms as these belong to $\mathfrak{K}_+$. The remaining terms
yield a tensor which is anti-symmetric in the first indices and anti-symmetric in the last indices. Thus automatically things of
the form
$e^{**12}$ or
$e^{**34}$ belong to
$\mathfrak{K}_+$ and don't need to be worried about.
Thus the only terms which matter are the following:

$$\begin{array}{|r|r|r|r|r|}
\noalign{\hrule}\text{Term}&\text{Coef}&\text{Coef}&\text{Coef}&\text{Contribution}\cr
\noalign{\hrule}e^{1223}&J=-1&R=2&T=1&-a+3b\vphantom{\vrule height 11pt}\cr
\noalign{\hrule}e^{1214}&K=-1&R=2&W=1&-a+3b\vphantom{\vrule height 11pt}\cr
\noalign{\hrule}e^{3423}&H=-1&R=2&U=1&-a+3b\vphantom{\vrule height 11pt}\cr
\noalign{\hrule}e^{3414}&G=-1&R=2&V=1&-a+3b\vphantom{\vrule height 11pt}\cr
\noalign{\hrule}
\end{array}$$
Thus we must have $-a+3b=0$ so we may take $a=3$ and $b=1$. This completes the proof in signature
$(0,4)$; the remaining cases are handled using the same techniques used in Section~\ref{sect-3.3}.

\section{The proof of Theorem~\ref{thm-5}}\label{sect-4}
Adopt the notation of Equation~(\ref{eqn-7}). Fix a bilinear form
$\varepsilon=(\varepsilon_{ij})$ on $\mathbb{R}^m$ which is $\pm$-invariant under $J_\pm$. 
Let ``$\circ$" denote symmetric tensor product. Let $\theta\in S^2_\mp\otimes S^2$. We form the germ of a pseudo-Riemannian metric which is
$\pm$-invariant under the action of $J_\pm$ by setting:
$$g=\varepsilon+\theta_{ijkl}x^kx^l dx^i\circ dx^j\,;$$
$g$ is a (para)-Hermitian metric on a neighborhood $\mathcal{O}$ of $0$ in $\mathbb{R}^m$. By
Theorem~\ref{thm-4}~(2) there is a unique Weyl connection $\nabla=\nabla(\theta)$ so that $(\mathcal{O},J_\pm,g,\nabla)$ is a
(para)-K\"ahler Weyl manifold. Let $\Theta(\theta):=R^\nabla(0)$;
$\Theta$ defines an equivariant linear map
$$\Theta:S^2_\mp\otimes S^2\rightarrow\mathfrak{K}_{\pm,\mathfrak{W}}\,.$$
To show that $\Theta$ is surjective and complete the proof of Theorem~\ref{thm-5}, we must
to show:
$$n_\eta(\operatorname{Range}(\Theta))=1\quad\text{for}\quad\eta\in
\{\pone,S_{0,\pm}^2,W_{\pm,3},\Lambda^2_{0,\mp},\Lambda^2_\pm\}\,.$$

\subsection{The representations $W_{\pm,i}$ for $i=1,2,3$}
Let $R^g(0)$ be the curvature of the Levi-Civita connection at the origin. The map $\mathcal{L}:\theta\rightarrow
R^g(0)$ is a linear function of $\theta$ given by:
$$(\mathcal{L}\theta)(x,y,z,w):= \theta(x,z,y,w)+\theta(y,w,x,z)-\theta(x,w,y,z)-\theta(y,z,x,w)\,.$$
We set $A:=\mathcal{L}(\theta)$. Similarly the map
$K_\pm:\Theta\rightarrow d\Omega^g$ is a linear map which takes $S^2_\mp(V^*)\otimes S^2(V^*)$ to $\Lambda^3(V^*)\otimes
V^*$. It is given by:
$$\{(K_\pm\Theta)(x,y,z)\}(w):=\Theta(x,J_\pm y,z,w)
+\Theta(y,J_\pm z,x,w)+\Theta(z,J_\pm x,y,w)\,.
$$
This shows that $\ker(K_\pm)$ is invariant under the action of $\mathcal{U}_\pm^\star$.
Clearly $\theta\in\ker(K_\pm)$ if and only if $g_\theta$ is a K\"ahler metric. On $\ker(K_\pm)$, we have
$\Theta=\mathcal{L}$ since $\phi=0$. Thus:
$$
\mathcal{L}:\ker(K_\pm)\rightarrow W_{1,\pm}\oplus W_{2,\pm}\oplus W_{3\pm}\,.
$$
Take
$$\Theta=\textstyle\frac12(e^1\otimes e^1\mp e^2\otimes e^2)\otimes (e^1\otimes e^1+e^2\otimes e^2)$$
so that the metric has the form
$$g_\Theta=\varepsilon+\frac12
(u_1^2+u_2^2)(du_1^2\mp du_2^2)\,.$$
The metric $g_\Theta$ is K\"ahler since it takes the form
$M_2\times\mathbb{C}$ where $M_2$ is a Riemann surface. Thus $\Theta\in\ker(K_\pm)$. Furthermore, the only
non-zero curvature components of the curvature tensor $A=R^g(0)$ at the origin, up to the usual $\mathbb{Z}_2$
symmetries, are given by
$$A(e_1,e_2,e_2,e_1)=1\,.$$

The symmetric Ricci tensor $\rho_s(x,y):=\frac12(\rho(x,y)+\rho(y,x))$ defines a map from
$\mathfrak{K}_{\pm,\mathfrak{W}}$ to $S^2$. We have
$$\rho_s(e_i,e_j)=\left\{\begin{array}{rl}
\varepsilon_{22}&\text{if }i=j=1,\\
\varepsilon_{11}&\text{if }i=j=2,\\
0&\text{otherwise}\end{array}\right\}\,.$$
Since $\rho_s$ is neither a multiple of $\langle\cdot,\cdot\rangle$ nor is $\rho_s$ perpendicular to
$\langle\cdot,\cdot\rangle$, $\rho_s$ has components both in $\pone$ and in $S^2_{0,\mp}$. Consequently
$$W_{\pm,1}\oplus W_{\pm,2}\subset\mathcal{L}(K_\pm)\,.$$

Let $S\in S_\mp^2$. Following \cite{TV81}, define:
\begin{eqnarray*}
S_1(x,y,z,w)&:=&\langle x,z\rangle S(y,w)+\langle y,w\rangle S(x,z)\\
&&-\langle x,w\rangle
S(y,z)-\langle y,z\rangle S(x,w)\\
S_2(x,y,z,w)&:=&2\langle x,J_\pm y\rangle S(z,J_\pm w)+2\langle z,J_\pm w\rangle S(x,J_\pm y)\\
&&+\langle x, J_\pm z\rangle S(y, J_\pm w)+\langle y, J_\pm w\rangle S(x, J_\pm z)\\
&&-\langle x, J_\pm w\rangle S(y, J_\pm z)-\langle y, J_\pm z\rangle S(x, J_\pm w)\,.
\end{eqnarray*}
Then the map $\Sigma: S\rightarrow S_1\mp S_2$ splits $\rho_s$
modulo a suitable normalizing constant.\footnote{This result was established in the positive definite
setting; it extends easily to the general context} We have:
$$\Sigma(\rho_s)(e_1,e_3,e_3,e_1)=-\varepsilon_{33}\varepsilon_{22}$$
and thus $\Sigma(\rho_s)$ is not a multiple of $R$ so $R$ has a
non-zero component in
$W_{\pm,3}$ and
$$W_{\pm,3}\subset\mathcal{L}(K_\pm)\,.$$

\subsection{The representations $\Lambda^2_\pm$ and $\Lambda^2_{0,\mp}$} 
The alternating part of the Ricci tensor, $\rho_a$ provides a map from
$\mathfrak{K}_{\pm,\mathfrak{W}}$ to $\Lambda^2$. If we can show $\rho_a\Theta$ is a surjective map to
$\Lambda^2_{0,\mp}\oplus\Lambda^2_{\pm}$, it will follow from Lemma~\ref{lem-2} that
$n_\eta(\mathfrak{K}_{\pm,\mathfrak{W}})\ge1$ which will complete the proof. We have that
$\phi$ is a multiple of $J_\pm^*\delta\Omega_\pm$ and that $d\phi$ is a multiple of $\rho_a$. Thus it
will suffice to give an example where $dJ_\pm^*\delta\Omega_\pm$ has components in both
$\Lambda^2_{0,\mp}$ and
$\Lambda^2_\pm$. Suppose $f(x)=x_1x_3$. Let:
$$ds^2:=\varepsilon_{11}e^{2f(x_1,x_3)}(dx^1\otimes dx^1\mp dx^2\otimes dx^2)
+\varepsilon_{22}(dx^3\otimes dx^3\mp dx^4\otimes dx^4)\,.$$
We have \cite{BGGH11}:
\begin{eqnarray*}
(\nabla^g\Omega_\pm)(\partial_{x_i},\partial_{x_j};\partial_{x_k})
     &=&\textstyle\frac12\{
g(\partial_{x_i},\partial_{x_k};J_\pm \partial_{x_j})
     -g(\partial_{x_j},\partial_{x_k};J_\pm \partial_{x_i})\\
&&+g(J_\pm\partial_{x_i},\partial_{x_k};\partial_{x_j})
     -g(J_\pm \partial_{x_j},\partial_{x_k};\partial_{x_i})\}\,.
\end{eqnarray*}
This permits us to compute that:
$$
(\nabla^{g}\Omega_\pm)(\partial_{x_1},\partial_{x_3};\partial_{x_k})=
\left\{\begin{array}{rll}
\mp\varepsilon_{11}e^{2f}\partial_{x_3}f&\text{if}&k=2\\
0&\text{if}&k\ne2
\end{array}\right\}\,.$$
The covariant derivative of the K\"ahler form has the symmetries \cite{BGGH11}:
\begin{eqnarray*}
(\nabla^g\Omega_\pm)(x,y;z)&=&-(\nabla^g\Omega_\pm)(y,x;z)
=\pm(\nabla^g\Omega_\pm)(J_\pm x,J_\pm y;z)
\\
&=&\mp
(\nabla^g\Omega_\pm)(x,J_\pm y;J_\pm z)\,.\vphantom{\vrule  height 12pt}
\end{eqnarray*}
It now follows that the non-zero components of $\nabla^{g}\Omega_\pm$ are given, up to the
$\mathbb{Z}_2$ symmetry in the first components, by:
\begin{eqnarray*}
&&(\nabla^{g}\Omega_\pm)(\partial_{x_1},\partial_{x_3};\partial_{x_2})
     =\mp\varepsilon_{11}e^{2f}\partial_{x_3}f,\\
&&(\nabla^{g}\Omega_\pm)(\partial_{x_1},\partial_{x_4};\partial_{x_1})
     =\pm\varepsilon_{11}e^{2f}\partial_{x_3}f,\\
&&(\nabla^{g}\Omega_\pm)(\partial_{x_2},\partial_{x_4};\partial_{x_2})
     =-\varepsilon_{11}e^{2f}\partial_{x_3}f,\\
&&(\nabla^{g}\Omega_\pm)(\partial_{x_2},\partial_{x_3};\partial_{x_1})
   =\pm\varepsilon_{11}e^{2f}\partial_{x_3}f\,.
\end{eqnarray*}
This then implies
\begin{eqnarray*}
&&J_\pm^*\delta\Omega_\pm=2\mp\partial_{x_3}f\cdot dx^3,\\
&&dJ_\pm^*\delta\Omega_\pm=2\mp\partial_{x_1}\partial_{x_3}f\cdot dx^1\wedge dx^3\,.
\end{eqnarray*}
This has components in both $\Lambda^2_{0,\mp}$ and in $\Lambda^2_\pm$. The desired result now follows.

\section*{Acknowledgements}
Research of both authors supported by project MTM2009-07756 (Spain), by project 174012 (Serbia), and by DFG
PI 158/4-6 (Germany).

\end{document}